\documentclass{amsart}

\usepackage{amsmath, amsthm, amssymb, amstext, amsfonts}
\usepackage{enumerate}
\usepackage{graphicx}
\usepackage[T1]{fontenc} 
\usepackage[applemac]{inputenc}

\theoremstyle{plain}

\newtheorem{thm}{Theorem}[section]
\newtheorem{lem}[thm]{Lemma}

\newtheorem*{prob*}{Problem}

\theoremstyle{definition}

\newtheorem{rem}[thm]{Remark}

\newcommand{\sm}{\ensuremath{\smallsetminus}}

\newcommand{\inv}{\ensuremath{^{-1}}}

\newcommand{\Aut}{\textnormal{Aut}}

\newcommand{\sub}{\subseteq}

\def\td{tree-decom\-po\-si\-tion}
\def\ta{tree amalgamation}
\def\qt{quasi-tran\-si\-tive}
\def\qi{quasi-iso\-metric}
\def\qiy{quasi-iso\-me\-try}
\def\qis{quasi-iso\-me\-tries}
\def\ble{bilipschitz equivalent}
\def\blec{bilipschitz equivalence}

\def\PWhom{unvarying}

\newcommand{\comment}[1]{}

\newcommand{\nat}{{\mathbb N}}

\def\?#1{\vadjust{\vbox to 0pt{\vss\vskip-8pt\leftline{%
     \llap{\hbox{\vbox{\pretolerance=-1
     \doublehyphendemerits=0\finalhyphendemerits=0
     \hsize16truemm\tolerance=10000\small
     \lineskip=0pt\lineskiplimit=0pt
     \rightskip=0pt plus16truemm\baselineskip8pt\noindent
     \hskip0pt        
     #1\endgraf}\hskip7truemm}}}\vss}}}

\newenvironment{txteq*}
  {
    \begin{equation*}
    \begin{minipage}[c]{0.85\textwidth} 
    \em                                
  }
  {\end{minipage}\end{equation*}\ignorespacesafterend}

\begin{document}

\title{Tree amalgamations and quasi-isometries}
\author{Matthias Hamann}
\thanks{Supported by the Heisenberg-Programme of the Deutsche Forschungsgemeinschaft (DFG Grant HA 8257/1-1).}
\address{Matthias Hamann, Alfr\'ed R\'enyi Institute of Mathematics, Hungarian Academy of Sciences, Budapest, Hungary}
\date{}

\begin{abstract}
We investigate the connections between \ta s and \qis.
In particular, we prove that the quasi-isometry type of multi-ended accessible \qt\ connected locally finite graphs is determined by the quasi-isometry type of their one-ended factors in any of their terminal factorisations.
Our results carry over theorems of Papsoglu and Whyte on \qis\ between multi-ended groups to those between multi-ended graphs.
In the end, we discuss the impact of our results to a question of Woess.
\end{abstract}

\maketitle

\section{Introduction}\label{sec_Intro}

Recently, the following theorem was proved in~\cite{HLMR}, which says that every multi-ended locally finite connected \qt\ graph splits over a finite subgraph as a \ta\ and which is a graph theoretic version for Stallings' splitting theorem of multi-ended groups, see~\cite{Stallings}.
(We refer to Section~\ref{sec_Prelim} for the definitions.)

\begin{thm}\label{thm_HLMR_Splitting}\cite[Theorem 5.3]{HLMR}
Every connected \qt\ locally finite\linebreak graph with more than one end is a non-trivial tree amalgamation with respect to the group actions of finite adhesion and finite identification length of two connected \qt\ locally finite graphs.\qed
\end{thm}

Just like Stallings' theorem enables us to prove theorems about multi-ended groups having knowledge of their factors, we are aiming for similar results for graphs.
Two such example were already proved in~\cite{HLMR}: \ta s respect hyperbolicity, i.\,e.\ two locally finite \qt\ graphs are hyperbolic if and only if their \ta\ is hyperbolic, see~\cite[Theorem 7.4]{HLMR}, and a connected locally finite \qt\ graph has only thin ends if and only if it has a terminal factorisation in finite graphs, see~\cite[Theorem 7.10]{HLMR}.
Here, a terminal factorisation can be seen as analogue of a terminal graph of groups: whenever we split our multi-ended graphs, we may ask if their factors have more than one end and apply to those our splitting theorem, too.
If we iterate this splitting and eventually have only factors with at most one end, the \emph{terminal factorisation} of the original graph consists of these finite or one-ended factors.
Kr\"on and M\"oller~\cite[Theorem 5.5]{KM-QuasiIsometries} proved that a connected locally finite \qt\ graph has only thin ends if and only if it is \qi\ to a tree.
Thus, the result of~\cite[Section 7.2]{HLMR} is the following.

\begin{thm}\label{thm_HLMR_ThinEnds}
	A connected \qt\ locally finite graph is \qi\ to a tree if and only if it has a terminal factorisation of only finite graphs.\qed
\end{thm}

In this paper we are looking more into the connections of \ta s with \qis.
Since \ta s of two graphs differ if we choose different adhesion sets (i.\,e.\ if the overlap of the two graphs differ in its size or its position within the graphs), the first natural question is whether the \qiy\ type of a \ta\ depends on these choices.
Our first theorem proves that (under some mild assumptions) it does not depend on these choices.

\begin{thm}\label{thm_PW0.2Graph}
Let $G_1$ and $G_2$ be infinite connected locally finite \qt\ graphs and $G$ and~$H$ be two non-trivial \ta s of~$G_1$ and~$G_2$ both having finite adhesion and finite identification.
Assume that some group of automorphisms of~$G_i$ acts \qt ly on~$G_i$ and on its adhesion sets for $i=1,2$.
Then $G$ is \qi\ to~$H$.
\end{thm}

The result stays essentially the same if we take the second \ta\ as a \ta\ of two graphs, one of which is \qi\ to~$G_1$ the other to~$G_2$.
More generally, our next result says that even iterated \ta s only depend on the \qiy\ type of infinite factors.

\begin{thm}\label{thm_PW0.3Graph}
Let $G$ and $H$ be locally finite quasi-transitive graphs with infinitely many ends and let $(G_1,\ldots,G_n),(H_1,\ldots,H_m)$ be factorisations of $G,H$, respectively.
If $(G_1,\ldots,G_n)$ and $(H_1,\ldots,H_m)$ have the same set of quasi-isometry types of infinite factors, then $G$ and $H$ are quasi-isometric.
\end{thm}

An obvious question is whether we can also obtain a reverse statement of Theorem~\ref{thm_PW0.3Graph}, i.\,e.\ if $G$ and~$H$ are \qi\ graphs, do all factorisations of~$G$ and~$H$ have the same set of quasi-isometry types of infinite parts?
In general this is false as an easy example shows: take a tree; this has a factorisation into finite graphs, but also the trivial factorisation into just itself.

In this example, we can still factorise the tree while this is not possible for the finite graphs.
So we might just ask for a reverse of Theorem~\ref{thm_PW0.3Graph} for terminal factorisations.
But not all connected locally finite \qt\ graphs have a terminal factorisation.
Those that do have one are the accessible connected locally finite \qt\ graphs, as was shown in~\cite{HLMR} and this is indeed a class of graphs where we obtain the reverse of Theorem~\ref{thm_PW0.3Graph} for terminal factorisation, as the following theorem shows.
(See Section~\ref{sec_Prelim} for the definition of accessibility.)

\begin{thm}
\label{thm_PW0.4Graph}
Let $G$ be an accessible locally finite quasi-transitive graph.
A locally finite quasi-transitive graph $H$ is quasi-isometric to~$G$ if and only if it satisfies the following conditions:
\begin{enumerate}[\rm (i)]
\item $H$ has the same number of ends as~$G$;
\item $H$ is accessible; and
\item any terminal factorisation of $H$ has the same set of quasi-isometry types of one-ended factors as any terminal factorisation of~$G$.
\end{enumerate}
\end{thm}

Papasoglu and Whyte~\cite{PW-QuasiIsometries} proved group theoretic versions of Theorems~\ref{thm_PW0.2Graph}, \ref{thm_PW0.3Graph} and~\ref{thm_PW0.4Graph}.
We carry over their proof ideas from free products with amalgamations and HNN-extensions of groups to the setting of \ta s of graphs.
We prove Theorems~\ref{thm_PW0.2Graph} and~\ref{thm_PW0.3Graph} in Section~\ref{sec_proof1} and Theorem~\ref{thm_PW0.4Graph} in Section~\ref{sec_proof2}.

In Section~\ref{sec_WoessQ}, we turn our attention towards a question of Woess that seems to be a bit unrelated at a first glance.
Woess~\cite[Problem 1]{Woess-Topo} posed the problem whether there are locally finite transitive graphs that are not \qi\ to some locally finite Cayley graph.
His problem was settled in the negative by Eskin et al.~\cite{EFW-DLnotTransitive} who proved that the Diestel-Leader graphs are counterexamples.
As a corollary of Theorem~\ref{thm_PW0.4Graph} we shall see that in order to find further counterexamples, it suffices to look at either one-ended graphs or inaccessible ones.

\section{Preliminaries}\label{sec_Prelim}

In this section, we state our major definitions an cite several results that we are going to use in the proofs of our main theorems.

Let $G$ and~$H$ be graphs.
A map $\varphi\colon V(G)\to V(H)$ is a \emph{\qiy} if there are constants $\gamma\geq 1$, $c\geq 0$ such that 
\[
\gamma\inv d_G(u,v)-c\leq d_H(\varphi(u),\varphi(v))\leq\gamma d_G(u,v)+c
\]
for all $u,v\in V(G)$ and such that $\sup\{d_H(v,\varphi(V(G)))\mid v\in V(H)\}\leq c$.
We then say that $G$ is \emph{quasi-isometric} to~$H$.
We call $G$ and~$H$ \emph{\ble} if they are \qi\ with $c=0$.

A tree is \emph{semiregular} or \emph{$(p_1,p_2)$-semiregular}, for $p_1,p_2\in\nat\cup\{\infty\}$, if for the canonical bipartition $V_1,V_2$ of its vertex set all vertices in~$V_1$ have the same degree~$p_1$ and all vertices in~$V_2$ have the same degree~$p_2$.

Let $G_1$ and $G_2$ be two graphs and let $T$ be a $(p_1,p_2)$-semiregular tree with canonical bipartition $V_1,V_2$ of its vertex set.
Let
\[c\colon E(T)\to\{(k,\ell)\mid 1\leq k\leq p_1,1\leq\ell\leq p_2 \}\]
such that for all $v\in V_i$ the $i$-th coordinates of the elements of $\{c(e)\mid v\in e\}$ are exhaust the set $\{k\mid 1\leq k\leq p_1\}$.
Note that in particular the $i$-th coordinates of the elements of $\{c(e)\mid v\in e\}$ are all distinct.
Let $\{S^i_k\mid 1\leq k\leq p_1\}$ be a set of subsets of~$V(G_i)$ such that all $S^i_k$ have the same cardinality.
Let $\varphi_{k,\ell}\colon S_k\to T_\ell$ be a bijection.

For each $v\in V_i$ with $i=1,2$, let $G_i^v$ be a copy of~$G_i$.
We denote by $S_k^v$ the copy of $S^i_k$ in $G_i^v$.
Let $H:=G_1+G_2$ be the graph obtained from the disjoint union of all $G_i^v$ by adding an edge between $x\in S_k^v$ and $\varphi_{k,\ell}(x)\in S_\ell^u$ for every edge $vu\in E(T)$ with $c(vu)=(k,\ell)$ and $v\in V_1$.
Let $G$ be the graph obtained from $H$ by contracting all new edges, i.\,e.\ all edges outside of the $G_i^v$.
We call $G$ the \emph{\ta} of $G_1$ and $G_2$ over~$T$ (with respect to the sets $S_k$ and the maps $\varphi_{k,\ell}$), denoted by $G_1\ast_T G_2$.
If $T$ is clear from the context, we cimply write $G_1\ast G_2$..
The sets $S_k$ and their copies in~$G$ are the \emph{adhesion sets} of the tree amalgamation. 
If the adhesion sets of a tree amalgamation are finite, then this tree amalgamation has \emph{finite adhesion}.
Let $\psi\colon V(H)\to V(G)$ be the canonical map that maps every $x\in V(H)$ to the vertex of~$G$ it ends up after all contractions.
A \ta\ $G_1\ast_T G_2$ is \emph{trivial} if there is some $G_i^v$ such that the restriction of~$\psi$ to~$G^i_v$ is bijective.
Note that a \ta\ of finite adhesion is trivial if $V(G_i)$ is the only adhesion set of~$G_i$ and $p_i=1$ for some $i\in\{1,2\}$.

The \emph{identification size}\footnote{Equivalently, as in~\cite{M-TreeAmalgamation}, you can define $G$ via identifications of vertices in the disjoint union of the $G_i^v$ instead of contraction of the newly added edges. From this point of view, we get the justification of the term `identification size'.} of a vertex $x\in V(G)$ is the smallest size of subtrees of~$T$ for which the contractions to obtain~$x$ take place, i.\,e.\ the size $n$ of the smallest subtree $T'$ of~$T$ such that $x$ is obtained by contracting only edges between vertices in $\bigcup_{u\in V(T')}V(G_j^u)$.
The \ta\ has \emph{finite identification} if every vertex has finite identification size.
The \emph{identification length} mentioned in Theorem~\ref{thm_HLMR_Splitting} is defined analogously using the diameter instead of the size of the subtrees over which the contractions take place.

\begin{rem}\label{rem_PsiQI}
For a \ta\ $G_1\ast G_2$ of finite identification, the canonical map $\psi\colon V(G_1+G_2)\to V(G_1\ast G_2)$ is a quasi-isometry whose constants depend only on the identification sizes of the vertices.
Furthermore, if there is in addition an upper bound on the diameters of adhesion sets, then $G_1$ embeds \qi ally into~$G_1+G_2$.
\end{rem}

If $G_1$ is finite we define the \emph{finite extension of~$G_2$ by~$G_1$ (with respect to $G_1\ast G_2$)} to be isomorphic to a subgraph of $G_1\ast G_2$ that is induced by a copy $G_2^v$ of~$G_2$ and all copies $G_1^u$ of~$G_1$ with $uv\in E(T)$.
It is straight forward to see that the finite extension of~$G_2$ is \qi\ to~$G_2$.

\begin{rem}\label{rem_FinExten}
Let $T$ be a $(p_1,p_2)$-semiregular tree and let $G_1\ast_T G_2$ be a \ta, where $G_1$ is finite.
Let $G_2'$ be the finite extension of~$G_2$ by~$G_1$ with respect to $G_1\ast_T G_2$.
We will define a \ta\ $G_2'\ast_{T'}G_2$ with $G_2'\ast_{T'}G_2=G_1\ast_T G_2$, where $T'$ will be a $(p_2(p_1-1),p_2)$-semiregular tree.
For that choose $u\in V_2$ and let $U_1$ be the set of vertices $v\in V_2$ with distance $0\textrm {\ mod\ } 4$ to~$u$ and $U_2:=V_2\sm U_1$.
To obtain $T'$ we start with $T$ and contract all edges of~$T$ that are incident with a vertex of~$U_1$.
The resulting graph is a $(p_2(p_1-1),p_2)$-semiregular tree.
We may assume that the vertex set of~$T'$ is $U_1\cup U_2$.
For each vertex in~$U_1$ we take a copy of~$G_2'$ and for each vertex in~$U_2$ we take a copy of~$G_2$.
There is a canonical way of assigning the labels to the edges of~$T'$ so that the \ta\ $G_2'\ast_{T'}G_2$ is the same as the \ta\ $G_1\ast_T G_2$.
\end{rem}

A \emph{factorisation} of a connected locally finite \qt\ graph $G$ is a tuple $(G_1,\ldots,G_n)$ of connected locally finite \qt\ graphs such that $G$ is obtained from the elements of the tuple by iterated \ta s of finite adhesion and finite identification such that for each step some group of automorphisms of each factors acts \qt ly on this factor and the set of its adhesion sets.
A factorisation is \emph{terminal} if every element of the tuple has at most one end.\footnote{We note that the definitions for (terminal) factorisations are weaker than in~\cite{HLMR} as we do not care about the specific group actions for our results here.}

We will use Theorem~\ref{thm_HLMR_Splitting} in a slightly different form to avoid the definition of a \ta\ with respect to the group actions.
We note that the statement about the finite identification varies a bit from~\cite[Theorem 5.3]{HLMR}, but the version we use here is mentioned in its proof in~\cite{HLMR}.

\begin{thm}\label{thm_HLMR_Splitting_v2}\cite[Theorem 5.3]{HLMR}
Every connected \qt\ locally finite\linebreak graph with more than one end is a non-trivial tree amalgamation $G_1\ast G_2$ of finite adhesion and finite identification of two connected \qt\ locally finite graphs such that the adhesion sets in each factor have at most two orbits under some group acting \qt ly on that factor.\qed
\end{thm}

We note that in a \ta\ $G_1\ast G_2$ such that all adhesion sets have bounded diameter in each factor, $G_i$ embeds \qi ally into $G_1\ast G_2$.

A stronger version of Theorem~\ref{thm_HLMR_Splitting_v2} was obtained in~\cite{HLMR} for accessible connected \qt\ locally finite graphs.
In order to define accessibility, we have to define ends of graphs first.
\emph{Rays} are one-way infinite paths and two rays in a graph $G$ are equivalent if they lie eventually in the same component of $G-S$ for every finite set~$S$.
The equivalence classes of rays are the \emph{ends} of~$G$.
A \qt\ graph is \emph{accessible} if there is some $n\in\nat$ such that every two ends can be separated by at most $n$ vertices.

\begin{thm}\label{thm_HLMR_Access}\cite[Theorem 6.2]{HLMR}
Let $G$ be an accessible connected \qt\ locally finite graph.
Then there are connected \qt\ locally finite graphs $G_1,\ldots, G_n$, $H_1,\ldots,H_{n-1}$ with $G=H_{n-1}$ and trees $T_1,\ldots, T_{n-1}$ such that the following hold:
\begin{enumerate}[\rm (1)]
\item every $G_i$ has at most one end;
\item for every $i\leq n-1$, the graph $H_i$ is a \ta\ $H\ast_{T_i}H'$ of finite adhesion, where
\[
H,H'\in\{G_j\mid 1\leq j\leq n\}\cup\{H_j\mid 1\leq j<i\}.
\]
\qed
\end{enumerate}
\end{thm}

Papasoglu and Whyte used a construction in~\cite{PW-QuasiIsometries} that we will use for our proofs, too.
However, we can express it in terms of \ta s and thus stick to our notations.
Let $G_1$ and~$G_2$ be graphs and let $v_i\in V(G_i)$ for $i=1,2$ be their \emph{base vertices}.
Let $(S^i_k)_{k\leq |G_i|}$ be the adhesion sets such that $(S^i_k)_{k\leq |G_i|}$ forms a partition of $V(G_i)$ into sets of size~$1$.
Assume that $S^1_1=\{v_1\}$ and $S^2_1=\{v_2\}$.
Let $T$ be a $(|G_1|,|G_2|)$-semiregular tree with canonical vertex partition $\{V_1,V_2\}$, and let $u\in V_1$.
Let
\[
c\colon E(T)\to\{(k,\ell)\mid 1\leq k\leq |G_1|,1\leq\ell\leq|G_2|\}
\]
be as required for a \ta\ with the following additional property for all edges $xy\in E(T)$, where $x$ is closer to~$u$ than~$y$:
\begin{enumerate}[(1)]
\item[$\bullet$] if $x\in V_1$, let the second coordinate of $c(xy)$ be~$1$;
\item[$\bullet$] if $x\in V_2$, let the first coordinate of $c(xy)$ be~$1$.
\end{enumerate}
We denote the graph $G_1+G_2$ by $G_1+_{v_1,v_2}G_2$.
Note that there is a unique edge in~$T$ with label $(1,1)$.
We call the corresponding edge in $G_1+_{v_1,v_2}G_2$ the \emph{base edge} of $G_1+_{v_1,v_2}G_2$.
It is the unique edge in $G_1+_{v_1,v_2}G_2$ that connects two base vertices.

Papasoglu and Whyte proved several lemmas about this special kind of \ta s.
Before we state them, need a further definition.
We call a graph $G$ \emph{\PWhom} if there is some $\gamma$ such that for every $u,v\in V(G)$ there is a $\gamma$-bilipschitz map $G\to G$ that maps $u$ to~$v$.
Note that every \qt\ graph is \PWhom.

\begin{lem}\label{lem_PW_1.1}\cite[Lemma 1.1]{PW-QuasiIsometries}
Let $G$ and~$H$ be \PWhom\ graphs.
Let $F$ be a graph such that there is some $\gamma>0$ such that the following hold.
\begin{enumerate}[{\rm (1)}]
\item $F$ contains families $(G_i)_{i\in I}$ and $(H_j)_{j\in J}$ of disjoint subgraphs such $V(F)$ is covered by them.
\item Every vertex of~$F$ is incident with a unique edge that lies outside all $G_i, H_j$.
\item Every edge of~$F$ outside of all $G_i,H_j$ is incident with a vertex in some $G_i$ and with a vertex in some $H_j$.
\item For every $i\in I$, $j\in J$ there is a $\gamma$-\blec\ $G_i\to G$, $H_j\to H$, respectively.
\item Contracting all edges inside the graphs $G_i$ and $H_j$ results in a tree.
\end{enumerate}
Then there is a constant $\delta$ (depending only on~$\gamma$ and the \PWhom ness-constants of~$G$ and~$H$) such that for any edge $e\in E(F)$ connecting some $G_i$ to some $H_j$ and any base points $u$ in~$G$ and $v$ in~$H$ there is a $\delta$-\blec\ $F\to G+_{u,v}H$ that maps $e$ to the base edge of $G+_{u,v}H$.\qed
\end{lem}

Lemma~\ref{lem_PW_1.1} is our main tool to switch from the \ta s we are stating with to \ta s that are defined using base vertices, in order to apply the two following lemmas.

Let $G_1$ and~$G_2$ be graphs with base vertices.
The \emph{wedge} of~$G_1$ and~$G_2$ is their disjoint union with an edge joining their base vertices.

\begin{lem}\label{lem_PW_1.3}\cite[Lemma 1.3]{PW-QuasiIsometries}
Let $G_1$ and~$G_2$ be infinite \PWhom\ graphs with base vertices $u,v$, respectively.
Then $G:=G_1+_{u,v} G_2$ is \ble\ to the wedge of any finite number of copies of~$G$.\qed
\end{lem}

\begin{lem}\label{lem_PW_1.4}\cite[Lemma 1.4]{PW-QuasiIsometries}
Let $G_1$ and~$G_2$ be infinite \PWhom\ graphs and $u,v$ their base vertices, respectively.
Then $G_1+_{u,v} G_2$ and $G_1+_{u,v}(G_2+_{v,v}G_2)$ are \ble.\qed
\end{lem}

We end this section by proving two \qiy\ results regarding infinitely-ended trees.

\begin{lem}\label{lem_QISemiTrees}
Every two locally finite trees with infinitely many ends and finitely many orbits of vertices are \qi.
\end{lem}

\begin{proof}
It suffices to prove that every locally finite tree $T$ with infinitely many ends and finitely many orbits of vertices is \qi\ to the $3$-regular tree.
To see this, we note that, if $P$ is a path with $k$ edges in a tree such that all vertices of~$P$ have degree~$3$ in that tree, then contracting $P$ to a single vertex leads to a vertex of degree $k+3$.
We fix a root $u$ in the $3$-regular tree $T_3$ and a root $v$ in~$T$.
Let $P_v$ be a path in~$T_3$ starting at~$u$ of length $d(v)-3$, where $d(v)$ denotes the degree of~$v$.
Now let us assume that for a subtree $T'$ of~$T$ that contains~$v$ we have the following:
\begin{enumerate}[(1)]
\item for each vertex $w$ of~$T'$ there is a path $P_w$ in~$T_3$;
\item distinct such paths are disjoint; and
\item $x,y\in V(T')$ are adjacent if and only if $T_3$ has an edge with one end vertex in~$P_x$ and and the other in~$P_w$.
\end{enumerate}
We pick a vertex $w\in T-T'$ that is adjacent to a vertex $x$ of~$T'$.
By~(3) there is a vertex~$a$ in~$T_3$ adjacent to~$P_x$.
Also (3) implies that the vertices in~$T_3$ that lie on the paths $P_y$ for $y\in V(T')$ form a subtree.
Hence, we find a path $P_w$ in~$T_3$ that starts at~$a$, is disjoint from any $P_y$ with $y\in V(T')$ and has length $d(w)-3$.
We end up with a collection of paths in~$T_3$, one for every vertex of~$T$, that partition $V(T_3)$.
Contracting these paths defines a \qiy, since the paths have bounded length, and results in a tree isomorphic to~$T$ due to the requirement on the lengths of the paths~$P_w$.
\end{proof}

The last result of this section is a sharpening of the easy direction of Theorem~\ref{thm_HLMR_ThinEnds}.

\begin{lem}\label{lem_HLMR_ThinEndsImproved}
A connected locally finite \qt\ graph that has a terminal factorisation of only finite graphs is \qi\ to a $3$-regular tree.
\end{lem}

\begin{proof}
Let $(G_1,\ldots,G_n)$ be a terminal factorisation of a connected locally finite \qt\ graph~$G$, where all $G_i$ are finite.
First, we proof by induction on~$n$ that $G$ is \qi\ to a tree with at most $n$ orbits on its vertex set.
This suffices to prove the assertion since such trees are \qi\ to a $3$-regular tree by Lemma~\ref{lem_QISemiTrees}.
If $n=2$, then the map $G_1+_T G_2\to T$ that maps the vertices in $G^u_i$ to~$u$ is a \qiy.
So by Remark~\ref{rem_PsiQI} and as the \ta\ of a factorisation is of finite identification, $G$ is \qi\ to~$T$.
Now let $n\geq 3$ and let $(H_1,H_2)$ be a factorisation of~$G$ such that $H_1$ and $H_2$ have terminal factorisations that are subsequences of $(G_1,\ldots, G_n)$.
Note that these are proper subsequences.
So by induction $H_i$ is \qi\ to some tree $T_i$ with finitely many orbits.
Let $\alpha_i\colon H_i\to T_i$ be a \qiy\ and let $T$ be a tree such that $G=H_1\ast_T H_2$.
For each $\Aut(G)$-orbit of $E(T)$ we choose an edge $e=uv$ in it and an edge $x^e_ux^e_v\in E(H_1+H_2)$ with $x^e_u\in H_i^u$ and $x^e_v\in H_j^v$.
We extend the definition of $x^e_u$ and $x^e_v$ to all edges of~$T$ in such a way that it is compatible with the action of $\Aut(G)$.
Now we construct a new tree~$T'$.
For that, we replace in~$T$ each vertex $u$ with a copy $T_i^u$ of~$T_i$ and replace the edge $uv\in E(T)$ by an edge $\alpha_i(x^{uv}_u)\alpha_j(x^{uv}_v)$.
By its construction and as $T$, $T_1$ and $T_2$ are trees, $T'$ is a tree, too.
There are only finitely many orbits on $V(T')$  as this is true for the other three involved trees and as the construction of~$T'$ respects the action of $\Aut(G)$.
Finally, we note that the \qis\ $\alpha_1$ and $\alpha_2$ extend to a \qiy\ $\alpha\colon G\to T'$ as the adhesion sets of $H_1\ast_T H_2$ have bounded diameter.
\end{proof}

\section{Tree amalgamations and \qis}\label{sec_proof1}

In this section, we shall prove our first main results, Theorems~\ref{thm_PW0.2Graph} and~\ref{thm_PW0.3Graph}.
In preparations for that, we prove that we may assume -- up to \qiy~-- that the \ta s that we consider have adhesion~$1$, disjoint adhesion sets and no vertices outside adhesion sets.

\begin{lem}\label{lem_AdhesionGood}
Let $G$ be a locally finite connected \qt\ graph and let $(G_1,G_2)$ be a factorisation of~$G$.
Then there is a locally finite connected \qt\ graph $H$ that has a factorisation $(H_1,H_2)$ such that the following hold.
\begin{enumerate}[{\rm (1)}]
\item\label{itm_AdhesionG} $G$ is \qi\ to~$H$;
\item\label{itm_AdhesionGi} $G_i$ is \qi\ to~$H_i$ for $i=1,2$;
\item\label{itm_AdhesionOne} $H_1\ast H_2$ has adhesion~$1$;
\item\label{itm_AdhesionDisjoint} all adhesion sets of $H_1\ast H_2$ are distinct;
\item\label{itm_AdhesionsCover} the adhesion sets of~$H_i$ cover $H_i$ for $i=1,2$.
\end{enumerate}
\end{lem}

\begin{proof}
We will modify the \ta\ $G_1\ast_TG_2$ and the involved graphs $G_1,G_2$ so that the resulting graphs and \ta\ satisfy our assertions.
We will do this step by step.
I.\,e., first we modify them so that (\ref{itm_AdhesionG})--(\ref{itm_AdhesionOne}) are are true, then modify the resulting so that (\ref{itm_AdhesionG})--(\ref{itm_AdhesionDisjoint}) are true and finally modify a last time to satisfy all five statements.

First, choose for each edge $uv\in E(T)$ such that $u$ gets replaced by a copy of~$G_1$ when moving to $G_1+G_2$ a vertex~$x_{uv}$ in the adhesion set in~$G_1^u$ belonging to this edge.
We do this so that our choices are invariant under $\Aut(G_1+G_2)$.
Now we delete in $G_1+G_2$ all edges between copies of~$G_1$ and copies of~$G_2$ except for those between $x_{uv}$ and $\varphi_{k,\ell}(x_{uv})$ where $(k,\ell)=c(uv)$.
Since all adhesion sets have bounded diameter, it follows that the identity map on $V(G_1+G_2)$ is a \qiy\ between these two graphs.
Contracting all edges between the copies of~$G_1$ and~$G_2$ leads to a \ta\ $F_1$ of~$G_1$ and~$G_2$ of adhesion~$1$ that is \qi\ to~$G$.
It follows from the choice of the vertices~$x_{uv}$ that $F_1$ is \qt\ and $(G_1,G_2)$ is a factorisation of~$F_1$.
So $F_1,G_1,G_2$ satisfy (\ref{itm_AdhesionG})--(\ref{itm_AdhesionOne}).

Now we replace each vertex~$x$ in~$G_1$ that lies in some adhesion set by copies $x_1,\ldots,x_{n_x}$, where $n_x$ denotes the number of adhesion sets in~$G_1$ that contain~$x$.
Note that $n$ is finite as the \ta\ has finite identification.
We add
\begin{enumerate}[(1)]
\item[$\bullet$] all edges $x_ix_j$ for $1\leq i,j\leq n_x$ with $i\neq j$ and for every $x$ in adhesion sets,
\item[$\bullet$] all edges $vx_i$ for every $1\leq i\leq n_x$ and for every edge $vx$ if $v$ lies in no adhesion set but $x$ lies in an adhesion set, and
\item[$\bullet$] all edges $x_iy_j$ for all $1\leq i\leq n_x$ and $1\leq j\leq n_y$ and edges $xy$ if $x\neq y$ lie in adhesion sets.
\end{enumerate}
Let $G_1'$ be the new graph.
Let $\alpha_1\colon G_1'\to G_1$ be the map that fixes all vertices of~$G_1'$ outside of adhesion sets and maps a vertex $x_i$ to its origin~$x$ otherwise.
Analogously, we define $G_2'$ for the graph~$G_2$ and a map $\alpha_2\colon G_2'\to G_2$.
It is easy to see that $\alpha_1$ and $\alpha_2$ are \qis.
Since $G_1$ and $G_2$ are \qt, so are $G_1'$ and $G_2'$.
Choose the adhesion sets of~$G_i'$ such that they are mapped by $\alpha_i$ to the adhesion sets of~$G_i$, have size~$1$, are disjoint and cover all vertices that get mapped into adhesion sets of~$G_i$ by~$\alpha_i$.
Set $F_2:=G_1'\ast G_2'$.
By construction, $F_2$ is \qt.
Obiously, we can extend the maps $\alpha_1$ and $\alpha_2$ to a map $\alpha\colon G_1'+G_2'\to G_1+G_2$ that is a \qiy.
By Remark~\ref{rem_PsiQI}, we obtain a \qiy\ $F_2\to G_1\ast G_2$.
Thus, $F_2, G_1', G_2'$ satisfy (\ref{itm_AdhesionG})--(\ref{itm_AdhesionDisjoint}).

Let $c_i\geq 0$ such that every vertex of~$G_i'$ has distance at most~$c_i$ to some vertex of the adhesion sets in~$G_i'$.
We define a new graph $H_i$ whose vertex set consists of the vertices of~$G_i'$ that lie in adhesion sets.
It has edges between every two vertices that have distance at most $2c_i+1$ in~$G_i'$.
Note that this condition ensures that $H_i$ is connected.
The identity maps $V(H_1)\to V(G_1')$ and $V(H_2)\to V(G_2')$ are \qis\ that extend to a \qiy\ $H_1+H_2\to G_1'+G_2'$.
Since $\Aut(G_1'), \Aut(G_2')$ acts \qt ly on the adhesion sets in~$G_1'$, in~$G_2'$, respectively, the analogue is true for $H_1$ and $H_2$ since their construction is compatible with the automorphisms.
Also, $H:=H_1\ast H_2$ is \qt\ and \qi\ to $H_1+H_2$ by Remark~\ref{rem_PsiQI}.
Then $H,H_1,H_2$ satisfy (\ref{itm_AdhesionG})--(\ref{itm_AdhesionsCover}).
\end{proof}

\begin{lem}\label{lem_taOneFactorFinite}
Let $G,H$ be connected graphs.
If $G$ is finite and $H$ infinite, then $G\ast H$ is \qi\ to $H\ast\ldots\ast H$, where the number of factors equals the number of adhesion sets of~$G$.
\end{lem}

\begin{proof}
Let $m$ be the number of adhesion sets in~$G$.
Consider the map that fixes in $G+ H$ all vertices in copies of~$H$ and that maps the vertices in copies of~$G$ to a some neighbour of that copy in a copy of~$H$.
This map is a \qiy\ $G+H\to H+\ldots+H$ with $m$ summands and its image is \qi\ to $H\ast \ldots \ast H$ with $m$ factors.
\end{proof}

\begin{thm}\label{thm_PW_0.1}
Let $F,G,H$ be connected \qt\ locally finite graphs and $T,T'$ semi-regular trees with infinitely many ends.
If $F$ and~$G$ are \qi, then any locally finite \ta s $F\ast_T H$ and $G\ast_{T'} H$ of finite adhesion and finite identification are \qi.
\end{thm}

\begin{proof}
By Lemma~\ref{lem_AdhesionGood}, we may assume that the adhesion sets have size~$1$, are disjoint and cover $F$, $G$, and $H$.

First, we consider the case that $F$ is finite. Then $G$ is finite, too.
If $H$ is finite, then with the notation used in the definition of \ta s we can map $F^v$ to~$v$ and $H^w$ to~$w$ to obtain a \qiy\ $F+H\to T$.
Similarly, $G+H$ is \qi\ to~$T'$.
Note that finite identification implies that $T$ and $T'$ are locally finite since $F,G$ and~$H$ are finite.
Since $T$ and $T'$ have infinitely many ends, they are \qi\ by Lemma~\ref{lem_QISemiTrees} and hence $F\ast_T H$ and $G\ast_{T'}H$ are \qi.

Now assume that $H$ is infinite, but $F$ is still finite.
By Lemma~\ref{lem_taOneFactorFinite}, $F\ast_T H$ is \qi\ to $H\ast\ldots\ast H$ with $|F|$ factors and $G\ast_{T'}H$ is \qi\ to $H\ast\ldots\ast H$ with $|G|$ factors.
By Lemma~\ref{lem_PW_1.1}, these are \qi\ to $H+_{v,v}\ldots+_{v,v}H$ and $H+_{v,v}\ldots+_{v,v}H$, respectively, for some $v\in V(H)$.
By Lemma~\ref{lem_PW_1.4}, we know that these are \qi, and hence $F\ast_T H$ is \qi\ to $G\ast_{T'} H$.

Let us now consider the case that $F$ and thus also~$G$ are infinite.
We first assume that $H$ is infinite.
We apply Lemma~\ref{lem_PW_1.1} and Lemma~\ref{lem_PW_1.4} to see that $F\ast H$ is \qi\ to $F\ast H\ast H$ and $G\ast H$ is \qi\ to $G\ast H\ast H$.
By replacing $H$ by $H\ast H$, we may assume that $H$ is a non-trivial \ta.

Let $\varphi\colon F\to G$ be a \qiy.
Let $B$ be the image of~$\varphi$ and let $A\sub V(F)$ such that the restriction of~$\varphi$ to~$A$ is a bijection $A\to B$.
Note that the vertices of~$F$, of~$G$ have bounded distance to~$A$, to~$B$, respectively, as $\varphi$ is a \qiy.

We consider the graph $F+_{u,v}H$, where $u\in A$ is the base vertex of~$F$ and $v\in B$ is the base vertex of~$H$.
Let $\alpha$ map vertices inside copies $F^t$ of~$F$ to copies of~$A$ such that $\sup\{d(x,\alpha(x))\mid x\in \bigcup F^t\}$ is finite and such that every base vertex is the image of itself but of no other vertex.
Note that $\alpha$ is finite-to-one as the graphs are locally finite.

Let us modify $F+_{u,v}H$.
We replace every edge that is incident with a vertex $x$ in a copy of~$F$ and a vertex $y$ in a copy of~$H$ by a new edge $\alpha(x)y$.
Since $d(x,\alpha(x))$ is bounded, the new graph $X$ is \qi\ to $F+_{u,v}H$ and thus it is \qi\ to $F\ast H$ by Lemma~\ref{lem_PW_1.1}.
We equip $A$ with a graph structure by adding edges between any two vertices with distance at most $\sup\{2d(x,\alpha(x))\mid x\in\bigcup F^t\}$.
As noted earlier, this results in a connected graph.
Then $A$ with this new metric is \ble\ to $A$ with the metric induced by~$F$.
We change $C$ accordingly, i.\,e.\ we replace every copy of~$F$ by a copy of~$A$ with the new edges, and obtain a graph~$X'$.
Extending the map $\alpha$ by the identity on the copies of~$H$, we obtain a quasi-isometry $F+H\to X'$.

We change $X$ in the following way: for each $a$ in a copy of~$A$ we choose a neighbour $u_a$ outside of the copies of~$A$ and we replace all edges $au$ with $u$ outside of copies of~$A$ by $uu_a$.
Let $Y$ be the resulting graph.
By the choice of $\alpha$, the base vertex of every copy of~$F$ has a unique neighbour outside of its copy of~$F$.
It follows that every component of~$Y$ with all copies of~$A$ deleted is a wedge of finitely many copies of~$H$.
As $H$ is a non-trivial \ta, Lemma~\ref{lem_PW_1.3} shows that each of the components of~$Y$ with the copies of~$A$ removed is \ble\ to~$H$.
So Lemma~\ref{lem_PW_1.1} implies that $F\ast H$ is \qi\ to $A\ast_{u,v} H$.
Analogously, $G\ast H$ is \qi\ to $B\ast_{w,v} H$, where $w\in B$ is the base vertex of~$G$.
Since $A$ and~$B$ are \ble, Lemma~\ref{lem_PW_1.1} implies that $A\ast_{u,v} H$ and $B\ast_{w,v} H$ are \qi\ and so are $F\ast H$ and $G\ast H$.

If $H$ is finite, then Lemma~\ref{lem_taOneFactorFinite} implies that $F\ast H$ is \qi\ to $F\ast\ldots\ast F$ with $|H|$ copies and $G\ast H$ is \qi\ to $G\ast \ldots\ast G$ with $|H|$ copies.
Then Lemmas~\ref{lem_PW_1.1} and~\ref{lem_PW_1.4} imply that $F\ast H$ is \qi\ to $F\ast F$ and $G\ast H$ is \qi\ to $G\ast G$.
The previous case with both $F$ and~$H$ infinite shows that $F\ast F$ is \qi\ to $F\ast G$ which in turn is \qi\ to $G\ast G$, which completes the proof.
\end{proof}

Let us prove Theorem~\ref{thm_PW0.2Graph} with the aid of Theorem~\ref{thm_PW_0.1}.

\begin{proof}[Proof of Theorem~\ref{thm_PW0.2Graph}.]
Let $G=G_1\ast_TG_2$ and $H=G_1\ast_{T'}G_2$.
Since $G_i$ is infinite and some group acts \qt ly on its vertices and its adhesion sets, it follows from~\cite[Proposition 4.6]{HLMR} and the connection between \ta s and \td s as discussed in~\cite[Section 5]{HLMR} that all degrees of~$T$ are infinite.
As the \ta s are non-trivial, both $T$ and~$T'$ have infinitely many ends.
Thus, the assertion follows from Theorem~\ref{thm_PW_0.1}.
\end{proof}

Before we turn our attention to Theorem~\ref{thm_PW0.3Graph}, we prove a small lemma.

\begin{lem}\label{lem_taFinInfin}
Let $G$ be an infinite locally finite \qt\ connected graph and let $(G_1,G_2)$ be a factorisation of~$G$ such that $G_1$ is infinite, $G_2$ is finite and the amalgamating tree is not a star.
Then $G$ is \qi\ to $G_1\ast G_1$.
\end{lem}

\begin{proof}
By Lemma~\ref{lem_taOneFactorFinite}, $G$ is \qi\ to $G_1\ast\ldots\ast G_1$ with as many factors as there are adhesion sets in~$G_2$.
As the amalgamating tree is not a star, this number is at least~$2$.
Applying Theorem~\ref{thm_PW_0.1} repeatedly implies that $G$ is \qi\ to $G_1\ast G_1$.
\end{proof}

We are going to prove a slightly more technical version of Theorem~\ref{thm_PW0.3Graph} that implies it trivially.

\begin{thm}\label{thm_PW0.3GraphTechnical}
Let $G$ be a locally finite \qt\ graph with infinitely many ends and let $(G_1,\ldots, G_m)$ be a factorisation of~$G$.
Let $H_1,\ldots, H_n$ be infinite factors in $(G_1,\ldots G_m)$ consisting of one representatives per infinite \qiy\ type.
Then one of the following is true.
\begin{enumerate}[{\rm (1)}]
\item If $H:=H_1\ast,\ldots\ast H_n$ has infinitely many ends, then $G$ is \qi\ to~$H$.
\item If $H$ does not have infinitely many ends, but there exists a finite graph~$H_{\rm fin}$ such that some non-trivial \ta\ $H\ast H_{\rm fin}$ has infinitely many ends, then $H\ast H_{\rm fin}$ is \qi\ to~$G$.
\item $G$ is \qi\ to a $3$-regular tree.
\end{enumerate}
\end{thm}

\begin{proof}
We prove the assertion by induction on the number $n$ of \qiy\ types of~$(G_1,\ldots,G_m)$.
If $n=0$, then $G$ is \qi\ to a $3$-regular tree by Lemma~\ref{lem_HLMR_ThinEndsImproved}.
Let $n=1$.
We distinguish the cases whether $(G_1,\ldots,G_m)$ has one or more than one infinite factor.
Let us first assume that $(G_1,\ldots,G_m)$ has only one infinite factor, say $G_1$.
If it has no finite factor, the assertion follows immediately.
So we may assume that $m\geq 2$.
We may assume $G=(\ldots(G_1\ast G_2)\ast\ldots\ast G_m)$.
Repeatedly applying Lemma~\ref{lem_taFinInfin} shows that $G$ is \qi\ to $G_1\ast \ldots\ast G_1$ with at least two factors.
Applying Lemma~\ref{lem_PW_1.1} and then Lemma~\ref{lem_PW_1.4} repeatedly, implies that $G$ is \qi\ to $G_1+_{u,u}G_1$ for the base vertex $u$ of~$G_1$.
Now if $G_1$ does not have infinitely many ends, we obtain analogously that for any finite graph $H_{\rm fin}$ and non-trivial \ta\ $G_1\ast H_{\rm fin}$ we have that $G_1\ast H_{\rm fin}$ is \qi\ to $G_1\ast G_1$, which shows the assertion in this case.
If $G_1$ has infinitely many ends, then it has a factorisation $(G_1^1,G_1^2)$.
If both factors are infinite, we can apply Lemmas~\ref{lem_PW_1.1} and~\ref{lem_PW_1.4} to see that $G_1+_{u,u}G_1$ is \qi\ to $G_1$, since Lemma~\ref{lem_PW_1.4} shows that $(G_1^1+_{u,v}G_1^2)+_{u,u}(G_1^1+_{u,v}G_1^2)$ is \qi\ to $G_1^1+_{u,v}G_1^2$.
If $G_1^1$ is infinite but $G_1^2$ is finite, $G_1$ is \qi to $G_1^1\ast G_1^1$ by Lemma~\ref{lem_taFinInfin}.
But then Lemmas~\ref{lem_PW_1.1} and~\ref{lem_PW_1.4} show that $G_1^1\ast G_1^1$ is \qi\ to $(G_1^1\ast G_1^1)+_{u,u}(G_1^1\ast G_1^1)$, which is \qi\ to $G_1\ast G_1$ by Lemma~\ref{lem_PW_1.1}.
If $G_1^1$ and $G_1^2$ are finite, then $G_1$ is \qi\ to a $3$-regular tree by Lemma~\ref{lem_HLMR_ThinEndsImproved} and so is $G_1\ast G_1$.
So by Lemma~\ref{lem_QISemiTrees}, $G_1$ is \qi\ to $G_1\ast G_1$.
Thus, we have shown the assertion in the case that $(G_1,\ldots,G_m)$ has only one infinite factor.

So let us assume that $(G_1,\ldots,G_m)$ has more than one infinite factor.
Since $n=1$, all of them are in the same \qiy\ class.
By Theorem~\ref{thm_PW_0.1} and Lemma~\ref{lem_taFinInfin}, we may assume that $G=H_1\ast\ldots\ast H_1$.
Lemmas~\ref{lem_PW_1.1} and~\ref{lem_PW_1.4} implies that $G$ is \qi\ to $H_1\ast H_1$.
If $H_1$ is one-ended, let $H_{\rm fin}$ be a finite graph and $T$ be a semi-regular tree with infinitely many ends such that $H_{\rm fin}\ast_T H_1$ has infinitely many ends.
It follows that $H_{\rm fin}\ast_T H_1$ is non-trivial.
By Lemma~\ref{lem_taFinInfin}, $H_{\rm fin}\ast_T H_1$ is \qi\ to $H_1\ast H_1$.
So $H_{\rm fin}\ast_T H_1$ is \qi\ to~$G$.

If $H_1$ has more than one end, then it splits as a \ta\ $H^1\ast H^2$ by Theorem~\ref{thm_HLMR_Splitting}.
If $H^1$ and $H^2$ are finite, then $H=H_1$ is quasi-isometric to a $3$-regular tree and so is~$G$ by Lemma~\ref{lem_HLMR_ThinEndsImproved}.
Hence, $G$ is quasi-isometric to~$H$.
If $H^1$ and~$H^2$ are infinite, then $H_1=H^1\ast H^2$ is \qi\ to $(H^1\ast H^1)\ast (H^2\ast H^2)=H_1\ast H_1$ by Lemmas~\ref{lem_PW_1.1} and~\ref{lem_PW_1.4}.
Theorem~\ref{thm_PW_0.1} implies that $G$ is \qi\ to $G_{\rm fin}\ast H_1$ for any finite $G_{\rm fin}$ such that $G_{\rm fin}\ast H_1$ is non-trivial.
If $H^1$ is finite and $H^2$ is infinite, then Lemma~\ref{lem_taFinInfin} implies that $H_1$ is \qi\ to $H^2\ast H^2$.
So Lemmas~\ref{lem_PW_1.1} and~\ref{lem_PW_1.4} imply that $H_1$ is \qi\ to $H_1\ast H_1$.
As $G$ is \qi\ to $H_1\ast H_1$, it is \qi\ to~$H_1$.

Let us now assume that $n\geq 2$.
For $i=0,\ldots, n$, let $G_i^1,\ldots, G_i^{k_i}$ be the factors of $(G_1,\ldots,G_m)$ that are \qi\ to~$H_i$, where $H_0$ is any finite graph.
Then
\[
G=G_0^1\ast\ldots\ast G_0^{k_0}\ast \ldots\ast G_n^1\ast\ldots\ast G_n^{k_n}.
\]
If $n\geq 3$ or $k_1\geq 2$, then $G':=G_1^1\ast\ldots\ast G_{n-1}^{k_{n-1}}$ has infinitely many ends and by induction it is \qi\ to either $H_1\ast\ldots \ast H_{n-1}$ or $G_{\rm fin}\ast H_1$ for any finite $G_{\rm fin}$ such that $G_{\rm fin}\ast H_1$ is non-trivial.
Lemma~\ref{lem_taFinInfin} shows that we can replace each finite factor by~$H_1$ and thus $G'$ is \qi\ to either $H_1\ast\ldots\ast H_{n-1}$ or $H_1\ast H_1$.
By Theorem~\ref{thm_PW_0.1} we may assume that $G_n^j=H_n$ for every $1\leq j\leq k_n$.
By applying Lemmas~\ref{lem_PW_1.1} and~\ref{lem_PW_1.4} we reduce the multiple factors of~$H_1$ and~$H_n$ to just one each and obtain that $G$ is \qi\ to either $H$ or $H_1\ast H_1\ast\ldots\ast H_n$, where another application of Lemmas~\ref{lem_PW_1.1} and~\ref{lem_PW_1.4} shows that $H_1\ast H_1\ast\ldots\ast H_n$ is \qi\ to~$H$.
\end{proof}

\section{Accessibility and \qis}\label{sec_proof2}

In this section we shall prove Theorem~\ref{thm_PW0.4Graph}.
The first step is to see that if a one-ended connected \qt\ locally finite graph embeds \qi ally into a \ta, then it embeds already \qi ally into one of the factors.

\begin{lem}\label{lem_qiEmbedsIntoFactor}
Let $G$ and~$H$ be connected \qt\ locally finite graphs and let $(G_1,G_2)$ be a factorisation of~$G$.
If $H$ has precisely one end and embeds \qi ally into~$G$, then it embeds \qi ally into either $G_1$ or~$G_2$.
\end{lem}

\begin{proof}
Let $\varphi\colon H\to G$ be a $(\gamma,c)$-\qi\ embedding.
Let $S$ be an adhesion set in~$G$ and let $S'$ be the set of vertices of~$G$ of distance at most $\gamma+c$.
If there are vertices of $\varphi(H)$ in different components of $G-S'$, then their preimages are not connected in $H-\varphi\inv(S')$ by the choice of~$S'$ and as $H$ is one-ended and $S'$ finite, there is only one infinite component $C_S^\infty$ of $H-\varphi\inv(S')$ and only finitely many finite components.
So we find $\Delta_S$ such that the images of all vertices of finite components of $H-\varphi\inv (S')$ have distance at most $\Delta_S$ from~$S$.
Note that since we have only finitely many orbits of adhesion sets, the numbers $\Delta_S$ are globally bounded by some~$\Delta$.

Let $uv$ be an edge of~$T$.
Then $G_i^u\cap G_j^v$ is an adhesion set.
The infinite component of $C_{G_i^u\cap G_j^v}^\infty$ gets mapped into either
\[
G_{T_u}:=\bigcup_{a\in V(T_u)} G_i^a\qquad \text{or}\qquad
G_{T_v}:=\bigcup_{a\in V(T_v)} G_i^a
\]
but not into both, where $T_w$ is the component of $T-uv$ that contains $w$ for $w\in\{u,v\}$.
We orient the edge $uv$ towards $u$ if the infinite component gets mapped into $G_{T_u}$ and we orient it towards $v$ otherwise.
It is easy to see that every vertex has at most one out-going edge and that there is at most one vertex without outgoing edges in this orientation of~$T$.
To see that there is at least one sink, let us suppose that this is not the case.
Then there is a directed ray in the orientation of~$T$.
As $G_1\ast G_2$ has finite identification, every $\varphi(a)$ for $a\in V(H)$ lies eventually outside the adhesion sets on that ray and for every adhesion set on that ray, we find a later one that is disjoint from it.
Obviously, there is an adhesion set~$S$ on that ray separating $\varphi(a)$ from $C_S^\infty$.
But this is impossible as $a$ must lie eventually within the $\Delta$-neighbourhood of all later adhesion sets of that ray.
So $T$ has a unique sink $G_i^x$ such that every vertex of~$H$ gets mapped by~$\varphi$ into the $\Delta$-neighbourhood of~$G_i^x$.
We can easily modify $\varphi$ and obtain a maps $\varphi'$ that maps $H$ \qi ally into~$G_i^x$ with respect to the metric of~$G$.
But since $G_i^x$ has only finitely many orbits of adhesion sets, $\varphi'$ also maps $H$ \qi ally into~$G_i^x$ with the metric of~$G_i^x$.
\end{proof}

Now we are able to prove Theorem~\ref{thm_PW0.4Graph}.

\begin{proof}[Proof of Theorem~\ref{thm_PW0.4Graph}.]
Let us first assume that $H$ satisfies (i) to (iii).
If $G$ and $H$ have infinitely many ends, then it follows directly from Theorem~\ref{thm_PW0.3Graph} that $G$ is \qi\ to~$H$.
If they have two ends, then no factor can be one-ended, so all are finite and hence $G$ and $H$ are \qi\ to the double ray and thus \qi\ to each other.
If they have one end, then each of the two terminal factorisations has at most one one-ended factor and all others are finite.
Moreover, a \ta\ of a one-ended graph and a finite graph is one-ended only if the non-trivial amalgamation tree is a star, i.\,e.\ a tree of diameter at most~$2$.
For such \ta s the finite extension of the one-ended factor by the finite factor is just the \ta\ itself and thus the \ta\ is \qi\ to the one-ended factor by Remark~\ref{rem_FinExten}.
Thus, $G$ is also \qi\ to~$H$ in this case.

Now let us assume that $G$ is \qi\ to~$H$.
Since accessibility and numbers of ends are invariant under quasi-iso\-metries, it only remains to prove (iii).
By Theorem~\ref{thm_HLMR_Access}, there exists a terminal factorisation $(G_1,\ldots, G_n)$ of~$G$.
Then there exist $F_1,\ldots,F_{n-1}$ such that $G=F_{n-1}$ and such that for every $i\leq n-1$ the graph $F_i$ is a \ta\ of finite adhesion with factors in
\[
\{G_j\mid 1\leq j\leq n\}\cup\{F_j\mid 1\leq j<i\}.
\]
Let $H'$ be a factor in a terminal factorisation $(H_1,\ldots,H_m)$ of~$H$.
By Remark~\ref{rem_PsiQI}, $H'$ maps \qi ally into~$H$ and thus into~$G$.
Applying Lemma~\ref{lem_qiEmbedsIntoFactor} recursively, we conclude that for some $G_i$ there is a \qi\ embedding $\varphi\colon H'\to G_i$.
Similarly, $G_i$ embeds \qi ally into some factor $F$ of the terminal factorisation $(H_1,\ldots, H_m)$ of~$H$ by a map $\psi$ with $\psi\circ\varphi=id$.
Since $\psi\circ\varphi=id$, we know that $F$ must be mapped by $\psi\circ\varphi$ into~$H'$.
As both are one-ended, we conclude $F=H'$.
Thus, $H'$ is \qi\ to~$G_i$.
\end{proof}

\section{Quasi-isometries between transitive graphs and Cayley graphs}\label{sec_WoessQ}

Woess~\cite[Problem 1]{Woess-Topo} asked whether there are transitive locally finite graphs that are not \qi\ to some locally finite Cayley graph.
Eskin et al.~\cite{EFW-DLnotTransitive} showed that the Diestel-Leader graphs are examples of transitive graphs that are not \qi\ to some locally finite Cayley graph.
Since the Diestel-Leader graphs are one-ended, the question arises what can be said about Woess' question for graphs that need not have one-ended graphs as building blocks in our \ta\ sense, i.\,e.\ inaccessible graphs.
Dunwoody~\cite{D-AnInaccessibleGraph} constructed an inaccessible locally finite transitive graphs that is another example for a negative answer to Woess' question.
As an application of our previous results, we obtain that one-ended and inaccessible examples are basically the only ones you have to consider when understanding the \qiy\ differences between locally finite transitive graphs and locally finite Cayley graphs.

\begin{thm}\label{thm_WQ}
Let $G$ be a locally finite transitive accessible graph that is not \qi\ to some locally finite Cayley graph.
Then there is a one-ended locally finite transitive graph that is \qi\ to some factor in  a terminal factorisation  of~$G$ and that is not \qi\ to some Cayley graph.
\end{thm}

\begin{proof}
If $G$ has precisely two ends, then it is \qi\ to the double ray by Theorem~\ref{thm_HLMR_ThinEnds}, which is a locally finite Cayley graph.
So we may assume that $G$ has infinitely many ends.
Let $(G_1,\ldots, G_n)$ be a terminal factorisation of~$G$.
Note that every $G_i$ is \qt\ and thus \qi\ to some transitive locally finite graph, see e.\,g.\ Kr\"on and M\"oller~\cite[Theorem 5.1]{KM-QuasiIsometries}.
Suppose that every $G_i$ is \qi\ to some locally finite Cayley graph~$H_i$.
Then $(G_1,\ldots,G_n)$ and $(H_1,\ldots,H_n)$ have the same \qiy\ types of infinite factors.
Let $\Gamma_i=\left<S_i\mid R_i\right>$ be a group that has $H_i$ as Cayley graph.
Let $\Gamma$ be the free product of all $\Gamma_i$ with presentation $\left<S\mid R\right>$, where $S=\bigcup_{1\leq i\leq n}S_i$ and $R=\bigcup_{1\leq i\leq n}R_i$.
Then the Cayley graph $H$ of~$\Gamma$ with respect to $\left<S\mid R\right>$ is the \ta\ of all~$H_i$, i.\,e.\ $H=(\ldots(H_1\ast H_2)\ast\ldots)\ast H_n$.
By Theorem~\ref{thm_PW0.3Graph}, $G$ is \qi\ to~$H$, which is a contradiction.
\end{proof}

We can also restrict Woess' question to certain classes of graphs and groups.
If this class is invariant under \qis, factorisations and \ta s, then the proof of Theorem~\ref{thm_WQ} stays true for it.

One such class are the hyperbolic graphs and group:
a graph is \emph{hyperbolic} if there is some $\delta\geq 0$ such that for every three vertices $x,y,z$ and every three paths, one between each pair of $\{x,y,z\}$, each of the paths lies in the $\delta$-neighbourhood of the union of the other two paths; and a finitely generated group is \emph{hyperbolic} if it has a locally finite hyperbolic Cayley graph.

It follows from the definition that hyperbolicity is preserved under \qis.
By~\cite[Theorem 7.10]{HLMR}, hyperbolicity is also preserved under factorisations and \ta s.
Additionally, hyperbolic \qt\ locally finite graphs are accessible by~\cite[Theorem 4.3]{H-Accessibility}.
Thus, if there is a hyperbolic locally finite \qt\ graph that shows that Woess' question is false, then there is already a one-ended such graph.

\bibliographystyle{amsplain}
\bibliography{/Users/Matthias/Documents/Bibs/Bibs}

\providecommand{\bysame}{\leavevmode\hbox to3em{\hrulefill}\thinspace}
\providecommand{\MR}{\relax\ifhmode\unskip\space\fi MR }
\providecommand{\MRhref}[2]{%
  \href{http://www.ams.org/mathscinet-getitem?mr=#1}{#2}
}
\providecommand{\href}[2]{#2}
\begin{thebibliography}{1}

\bibitem{D-AnInaccessibleGraph}
M.J. Dunwoody, \emph{An {I}naccessible {G}raph}, Random walks, boundaries and
  spectra, Progr.\ Probab., vol.~64, Birkh\"auser/Springer Basel AG, 2011,
  pp.~1--14.

\bibitem{EFW-DLnotTransitive}
A.~Eskin, D.~Fisher, and K.~Whyte, \emph{Coarse differentiation of
  quasi-isometries {I}: {S}paces not quasi-isometric to {C}ayley graphs}, Ann.\
  of Math. (2) \textbf{176} (2012), no.~1, 221--260.

\bibitem{H-Accessibility}
M.~Hamann, \emph{Accessibility in transitive graphs}, Combinatorica \textbf{38}
  (2018), 847--859.

\bibitem{HLMR}
M.~Hamann, F.~Lehner, B.~Miraftab, and T.~R\"uhmann, \emph{A {S}talling's type
  theorem for quasi-transitive graphs}, in preparation, 2018.

\bibitem{KM-QuasiIsometries}
B.~Kr\"on and R.G. M\"oller, \emph{Quasi-isometries between graphs and trees},
  J.\ Combin.\ Theory (Series B) \textbf{98} (2008), no.~5, 994--1013.

\bibitem{M-TreeAmalgamation}
B.~Mohar, \emph{Tree amalgamation of graphs and tessellations of the {C}antor
  sphere}, J.\ Combin.\ Theory (Series B) \textbf{96} (2006), 740--753.

\bibitem{PW-QuasiIsometries}
P.~Papasoglu and K.~Whyte, \emph{Quasi-isometries between groups with
  infinitely many ends}, Comm.\ Math.\ Helv. \textbf{77} (2002), 133--144.

\bibitem{Stallings}
J.R. Stallings, \emph{Group theory and three dimensional manifolds}, Yale
  Math.\ Monographs, vol.~4, Yale Univ.\ Press, New Haven, CN, 1971.

\bibitem{Woess-Topo}
W.~Woess, \emph{Topological groups and infinite graphs}, Discrete Math.
  \textbf{95} (1991), no.~1-3, 373--384.

\end{thebibliography}

\end{document}